\documentclass[12pt]{article}
\usepackage{amsthm}
\usepackage{tikz}
\usepackage{fullpage}
\usepackage{float}
\usepackage{multirow}
\usepackage{amsmath}
\usepackage{amssymb}
\usepackage{amsthm}
\usepackage{amscd}
\usepackage{amsfonts}

\usepackage{comment}
\usepackage{hyperref}
\usepackage{url}
\usepackage{mleftright} 
\usepackage{float}
\usepackage{multirow}
\usepackage{hyperref}
\PassOptionsToPackage{lowtilde}{url}

\newtheorem{theorem}{Theorem}[section]

\newtheorem{corollary}[theorem]{Corollary}
\newtheorem{proposition}[theorem]{Proposition}

\newtheorem{question}[theorem]{Question}

\def\ev{\mathop{\rm ev}\nolimits}

\def\1{{\bf 1}}

\def\b{\mbox{\boldmath $b$}}
\def\c{\mbox{\boldmath $c$}}
\def\d{\mbox{\boldmath $d$}}
\def\m{\mbox{\boldmath $m$}}
\def\vu{\mbox{\boldmath $u$}}

\DeclareMathOperator{\dgr}{dgr}
\DeclareMathOperator{\rank}{rank}
\DeclareMathOperator{\spec}{sp}

\def\vecnu{\mbox{\boldmath $\nu$}}
\def\C{\mathbb C}
\def\Real{\mathbb R}
\def\Z{\mathbb Z}

\usepackage{notation}
\date{}
\begin{document}
\title{Optimization of eigenvalue bounds for the independence and chromatic number of graph powers}
\author{A. Abiad$^{a,b,c}$, G. Coutinho$^d$, M. A. Fiol$^e$, B. D. Nogueira$^d$, S. Zeijlemaker$^a$ \\
	$^{a}${\small Department of Mathematics and Computer Science} \\
	{\small Eindhoven University of Technology, Eindhoven, The Netherlands}\\
	$^{b}${\small Department of Mathematics: Analysis, Logic and Discrete Mathematics} \\
	{\small Ghent University, Ghent, Belgium}\\
	$^{c}${\small Department of Mathematics and Data Science} \\
	{\small Vrije Universiteit Brussels, Brussels, Belgium}\\
	{\small {\tt{\{a.abiad.monge,s.zeijlemaker\}@tue.nl}}}\\
	$^d${\small Department of Computer Science}\\
	{\small Federal University of Minas Gerais, Belo Horizonte, Brazil}\\
	{\small{\tt {\{gabriel,bruno.demattos\}@dcc.ufmg.br}}}\\
	$^{e}${\small Departament of Mathematics} \\
	{\small Polytechnic University of Catalonia, Barcelona} \\
	{\small Barcelona Graduate School of Mathematics}\\
	{\small{\tt{miguel.angel.fiol@upc.edu}}}}

\maketitle

\begin{abstract}
The $k^{\text{th}}$ power of a graph $G=(V,E)$, $G^k$, is the graph whose vertex set is $V$ and in which two distinct vertices are adjacent if and only if their distance in $G$ is at most $k$. 
This article proves various eigenvalue bounds for the independence number and chromatic number of $G^k$ which purely depend on the spectrum of $G$, together with a method to optimize them. Our bounds for the $k$-independence number also work for its quantum counterpart, which is not known to be a
computable parameter in general, thus justifying the use of integer programming to optimize them. Some of the bounds previously known in the literature follow as a corollary of our main results. Infinite families of graphs where the bounds are sharp are presented as well.
\end{abstract}

\section{Introduction}

For a positive integer $k$, the $k^{\text{th}}$ \emph{power of a graph} $G =(V, E)$, denoted by $G^k$, is a graph with vertex set $V$ in which two distinct elements of $V$ are joined by an edge if there is a path in $G$ of length at most $k$ between them. 
For a nonnegative integer $k$, a $k$-\emph{independent set} in a graph $G$ is a vertex set such that the distance between any two distinct vertices on it is bigger than $k$. Note that the $0$-independent set is $V(G)$ and an $1$-independent set is an independent set. The $k$-\emph{independence number} of a graph $G$, denoted by $\alpha_k(G)$, is the maximum size of a $k$-independent set in $G$. Note that $\alpha_k(G)=\alpha(G^k)$.

The $k$-independence number is an interesting graph-theoretic parameter that is closely related to coding theory, where codes relate to $k$-independent sets in Hamming graphs \cite[Chapter 17]{MWS}. The $k$-independence number of a graph is also directly related to the $k$-\emph{distance chromatic number}, denoted by $\chi_k(G)$, which is just the chromatic number of $G^k$. Hence, $\chi_k(G)=\chi(G^k)$. It is well known that $\alpha_1(G) = \alpha(G) \geq n/\chi(G)$. Therefore, lower bounds on the  $k$-distance chromatic number can be obtained by finding upper bounds on the corresponding $k$-independence number, and vice versa. The parameter $\alpha_k$ has also been studied in several other contexts (see \cite{AF2004, DZ2003,FG1998,FGY1996,FGY1997,N} for some examples) and it is related to other combinatorial parameters, such as the average distance \cite{fh97}, the packing chromatic number \cite{GHHHR2008}, the injective chromatic number \cite{HKSS2002}, the strong chromatic index \cite{M2000} and the $d$-diameter \cite{CDS}. Recently, the $k$-independence number has also been related to the beans function of a connected graph \cite{en}.

The study of the $k$-independence number has attracted quite some attention. Firby and Haviland \cite{fh97} proved an upper bound for $\alpha_k(G)$ in an $n$-vertex connected graph. In 2000, Kong and Zhao \cite{kz1993} showed that for every $k\geq 2$, determining $\alpha_k(G)$ is NP-complete for general graphs. They also showed that this problem remains NP-complete for
regular bipartite graphs when $k\in \{2, 3, 4\}$ \cite{kz2000}.  For each fixed integer $k \geq 2$ and $r\geq 3$, Beis, Duckworth
and Zito \cite{bdz2005} proved some upper bounds for $\alpha_k(G)$ in random $r$-regular graphs. O, Shi, and Taoqiu \cite{OShiTaoqiu2019} showed sharp upper bounds for the $k$-independence number in an $n$-vertex
$r$-regular graph for each positive integer $k\geq 2$ and $r\geq 3$. The case of $k=2$ has also received some attention: Duckworth and Zito \cite{DZ2003} showed a heuristic for finding a large 2-independent set of regular graphs, and Jou, Lin and Lin \cite{JLL2020} presented a sharp upper bound for
the 2-independence number of a tree.

Most of the existing algebraic work on bounding $\alpha_k$ is based on the following two classic results. Let $G$ be a graph with $n$ vertices and adjacency matrix eigenvalues $\lambda_1\ge \lambda_2\ge \cdots \ge \lambda_n$. The first well-known spectral bound (or {\em `inertia bound'\/}) for the independence number $\alpha=\alpha_1$ of $G$ is due to Cvetkovi\'c \cite{c71}:

\begin{equation}
\label{bound:cvetkovic}
\alpha\le \min \{|\{i : \lambda_i\ge 0\}| , |\{i : \lambda_i\le 0\}|\}.
\end{equation}

When $G$ is regular, another well-known bound ({\em `ratio bound'\/}) is due to Hoffman (unpublished):
\begin{equation}
\label{bound:hoffman}
\alpha \leq \frac{n}{1-\frac{\lambda_1}{\lambda_n}}.
\end{equation}


Abiad, Cioab\u{a}, and Tait \cite{act2016} obtained the first two spectral upper bounds for the $k$-independence number of a graph: an inertial-type bound and a ratio-type bound. They constructed graphs that attain equality for their first bound and showed that their second bound compares favorably to previous bounds on the $k$-independence number. Abiad, Coutinho, and Fiol \cite{acf2019} extended the  spectral bounds from \cite{act2016}. Wocjan, Elphick, and Abiad \cite{wea2020} showed that the inertial bound by Cvetkovi\'c is also an upper bound for the quantum $k$-independence number. Recently, Fiol \cite{fiol20} introduced the minor polynomials in order to optimize, for $k$-partially walk-regular graphs, a ratio-type bound.

In this article we present several sharp inertial-type and ratio-type bounds for $\alpha_k$ and $\chi_k$ which depend purely on the eigenvalues of $G$, and we propose a method to optimize such bounds using Mixed Integer Programming (MILP).
The fact that the inertial-type of bound that we consider is also valid to upper bound the quantum $k$-independence number [Theorem 7, \cite{wea2020}] justify the method we propose in this paper to optimize our bounds. It is not known whether quantum counterparts of $\alpha$ or $\chi$ are computable functions \cite{MancinskaRoberson}, and our bounds sandwich these parameters with the classical versions. 

If one wants to use the classical spectral upper bounds on the independence number (\ref{bound:cvetkovic}) and (\ref{bound:hoffman}) to bound $\alpha(G^k)=\alpha_k(G)$, one needs to know how the spectrum of $G^k$ relates to the spectrum of $G$. In the case when the spectrum of $G$ and $G^k$ are related, we show that previous work by Fiol \cite{f02} can be used to derive a sharp spectral bound for regular graphs which concerns the following problem posed by Alon and Mohar \cite{am2002}: among all graphs G of maximum degree at most $d$ and girth at least $g$, what is the largest possible value of $\chi(G^k)$? 

In general, though, the spectra of $G^k$ and $G$ are not related. We also prove various eigenvalue bounds for $\alpha_k$ and $\chi_k$ which only depend on the spectrum of $G$. In particular, our bounds are functions of the eigenvalues of $A$ and of certain counts of closed walks in $G$ (which can be written as linear combinations of the eigenvalues and eigenvectors of $A$). Under some extra assumptions (for instance, that of partial walk-regularity), we improve the known spectral inertial-type bounds for the $k$-independence number. Our approach is based
on a MILP implementation which finds the best polynomials that minimize the bounds. For some cases and some infinite families of graphs, we show that our bounds are sharp, and also in other cases that they coincide, in general, with the Lov\'asz
theta number.

\section{A particular case: the spectrum of $G^k$ and $G$ are related}

Our main motivation for this section comes from distance colorings, which have received a lot of attention in the literature. In particular, special efforts have been put on the following question of Alon and Mohar \cite{am2002}:

\begin{question}\label{que:AlonMohar}
	What is the
largest possible value of the chromatic number $\chi(G^k)$ of $G^k$, among all graphs $G$ with maximum degree at most $d$ and girth (the length of a shortest cycle contained in $G$) at least $g$?
\end{question}

The main challenge in Question \ref{que:AlonMohar} is to provide
examples with large distance chromatic number (under the condition of
girth and maximum degree). For $k = 1$, this question was essentially a long-standing problem of Vizing, one that stimulated much of the work on the chromatic number of bounded degree triangle-free graphs, and was eventually settled asymptotically by Johansson \cite{jo1996} by using the probabilistic method. The case $k = 2$ was considered and settled asymptotically by Alon and
Mohar \cite{am2002}.

The aim of this section is to show the first eigenvalue bounds on $\chi_k$ which concern Question \ref{que:AlonMohar} for regular graphs and when the spectrum of $G^k$ is related to the one of $G$. The spectra of $G$ and $G^k$ are related when the adjacency matrix of $G^k$ belongs to the algebra generated by the adjacency matrix of $G$, that is, there is a polynomial $p$ such that $p(A(G))=A(G^k)$. For instance, this happens when $G$ is $k$-partially distance polynomial \cite{ddfgg11}. In this framework, and when $\deg p=k$ (or, in particular, when $G$ is $k$-partially distance-regular \cite{ddfgg11}) we can use Proposition \ref{propo:f02} from \cite{f02} to derive spectral bounds. Before stating the results, we need to introduce some concepts and notations.

Let $G=(V,E)$ be a graph with $n=|V|$ vertices, $m=|E|$ edges, and adjacency matrix $A$ with  spectrum
$
\spec G =\spec A=\{\theta_0^{m_0},\theta_1^{m_1},\ldots,\theta_d^{m_d}\},
$
where the different eigenvalues are in decreasing order, $\theta_0>\theta_1>\cdots>\theta_d$, and the superscripts stand for their multiplicities (since $G$ is supposed to be connected, $m_0=1$).
When the eigenvalues are presented with possible repetitions, we shall indicate them by
$
\ev G:  \lambda_1 \geq \lambda_2 \geq \cdots\geq \lambda_n.
$
Let us consider the scalar product in $\Real_d[x]$:
\begin{equation}
\label{escalar-prod}
\langle f,g\rangle_G=\frac{1}{n}\tr(f(A)g(A))=\frac{1}{n}\sum_{i=0}^{d} m_i f(\theta_i)g(\theta_i).
\end{equation}
The so-called {\em predistance polynomials} $p_0(=1),p_1,\ldots, p_d$, which were introduced by Fiol and Garriga in \cite{fg97}, are a sequence of orthogonal polynomials with respect to the above product, with  $\dgr p_i=i$, and they are normalized in such a way that $\|p_i\|_G^2=p_i(\theta_0)$ for $i=0,\ldots,d$.
Therefore, they are uniquely determined, for instance, following the Gram-Schmidt process. These polynomials were used to prove the so-called `spectral excess theorem' for distance-regular graphs, where $p_0(=1),p_1,\ldots, p_d$ coincide with the so-called distance polynomials.

\begin{proposition} \cite{f02}
\label{propo:f02}
Let $G=(V,E)$ be a regular graph with $n$ vertices, spectrum $\spec G=\{\theta_0^{m_0},\theta_1^{m_1},\ldots,\theta_d^{m_d}\}$, and predistance polynomials $p_0,\ldots,p_d$. For a given integer $k\le d$ and a vertex $u\in V$, let $s_k(u)$ be the number of vertices at distance at most $k$ from $u$, and consider the sum polynomial $q_k=p_0+\cdots+p_k$. Then, $q_k(\theta_0)$ is bounded above by the harmonic mean $H_k$ of the numbers $s_k(u)$, that is
$$
q_k(\theta_0)\le H_k=\frac{n}{\sum_{u\in V}\frac{1}{s_k(u)}},
$$
and equality occurs if and only if $q_k(A)=I+A(G^k)$.
\end{proposition}

Since it is known that $q_k(\theta_0)\ge q_k(\theta_i)$ for $i=1,\ldots,d$, Proposition \ref{propo:f02} and the bounds \eqref{bound:cvetkovic}--\eqref{bound:hoffman} yield the following bounds on $\alpha_k$ and $\chi_k$:

\begin{corollary}
\label{coro-qk}
 Let $G$ be a regular graph with eigenvalues $\lambda_1\ge \cdots \ge \lambda_n$, satisfying $q_k(\lambda_1)= H_k$. Let $q'_k=q_k-1$, so that $A(G^k)=q'_k(A)$. Then,

  \begin{equation}
 \label{bound:cvetkovic-chik}
\chi_k\ge \frac{n}{\min \{|\{i : q'_k(\lambda_i)\ge 0\}| , |\{i : q'_k(\lambda_i)\le 0\}|\}},
\end{equation}
\begin{equation}
\label{bound:hoffman-chik}
\chi_k\ge \frac{n}{1-\frac{q'_k(\lambda_1)}{\min\{q'_k(\lambda_i)\}}},
\end{equation}

 and the corresponding upper bounds

\begin{equation}
\label{bound:cvetkovic-Gk}
\alpha_k  \le \min \{|\{i : q'_k(\lambda_i)\ge 0\}| , |\{i : q'_k(\lambda_i)\le 0\}|\},
\end{equation}
\begin{equation}
\label{bound:hoffman-Gk}
\alpha_k \leq 1-\frac{q'_k(\lambda_1)}{\min\{q'_k(\lambda_i)\}}.
 \end{equation}

\end{corollary}

Corollary \ref{coro-qk} provides the first two spectral bounds to Question \ref{que:AlonMohar} for regular graphs. This is due to the fact that another case where $A(G^k)=q_k(A)-I$ (that is, the spectrum of $G^k$ and $G$ are related) is when $G$ is $\delta$-regular graph with girth $g$ and $k=\lfloor \frac{g-1}{2}\rfloor$. In this situation, we know that $G$ is $k$-partially distance-regular with $a_i=0$ for $i\leq k$
\cite{adf2016} and hence $q_0=1$, $q_1=1+x$, and $q_{i+1}=xq_i-(\delta-1)q_{i-1}$ for $i=1,\ldots,k-1$.

Regarding Question \ref{que:AlonMohar}, Kang and Pirot \cite{kp2016} provide several upper and lower bounds for $k\geq 3$,
all of which are sharp up to a constant factor as $d \rightarrow \infty$. While their upper bounds
rely in part on the probabilistic method, their lower bounds are various
direct constructions whose building blocks are incidence structures. Actually, some tight examples for our bound (\ref{bound:hoffman-chik}) can be constructed from the latter. In particular, from even cycles using the balanced bipartite product `$\bowtie$' introduced in \cite{kp2016,kp2018}. Let $G_1 = (V_1=A_1\cup B_1,E_1)$ and $G_2 = (V_2=A_2\cup B_2,E_2)$ be bipartite graphs with $|A_1| = |B_1|$ and $|A_2| = |B_2|$, also known as balanced bipartite graphs. Assume vertex sets $A_i = \{a_1^i,\dots a_{n_i}^i\}$ and $B_i = \{b_1^i,\dots b_{n_i}^i\}$ be ordered such that $(a_j^i,b_j^i) \in E_i$ for $j = 1,2,\dots,n_i$. Then the product $G_1\bowtie G_2$ is defined as $\left(V_{G_1\bowtie G_2}, E_{G_1\bowtie G_2}\right)$ with
\begin{equation*}
\begin{array}{rll}
V_{G_1\bowtie G_2} &:= &A_1 \times A_2 \cup B_1 \times B_2 \\
E_{G_1\bowtie G_2} &:= &\lbrace \left((a_i^1,a^2), (b_i^1,b^2)\right) \mid i\in \{1,\dots,n_1\}, (a^2,b^2)\in E_2 \rbrace \cup \\
&&\lbrace \left((a^1,a_i^2), (b^1,b_i^2)\right) \mid i\in \{1,\dots,n_2\}, (a^1,b^1)\in E_1 \rbrace,
\end{array}
\end{equation*}
which is again a balanced bipartite graph. Moreover, if $G_1$ and $G_2$ are regular with degree $d_1$ and $d_2$, then $G_1\bowtie G_2$ is regular with degree $d_1+d_2-1$. The graphs $C_8\bowtie C_8$, $C_8\bowtie C_{12}$, $C_8\bowtie C_{16}$ and $C_{12}\bowtie C_{12}$, where $C_n$ denotes the cycle on $n$ vertices, each have girth 6 and satisfy Equation \eqref{bound:hoffman-Gk} with equality for $\alpha_2$. The bound \eqref{bound:hoffman-Gk} is also tight for several named Sage graphs, which are shown in Table \ref{tableremark}.

\begin{table}[H]
	\centering
\begin{tabular}{ |l|c c c| }
\hline
Name & Girth & $k$ & $\alpha_k$\\
\hline
Moebius-Kantor Graph & 6 & 2 & 4\\
Nauru Graph & 6 & 2 & 6\\
Blanusa First Snark Graph & 5 & 2 & 4\\
Blanusa Second Snark Graph & 5 & 2 & 4\\
Brinkmann graph & 5 & 2 & 3\\
Heawood graph & 6 & 2 & 2\\
Sylvester Graph & 5 & 2 & 6\\
Coxeter Graph & 7 & 3 & 4\\
Dyck graph & 6 & 2 & 8\\
F26A Graph & 6 & 2 & 6\\
Flower Snark & 5 & 2 & 5\\
\hline
\end{tabular}
 \caption{Named Sage graphs for which bound \eqref{bound:hoffman-Gk} from Corollary \ref{coro-qk} is tight.} \label{tableremark}
\end{table}

\section{The general case: the spectrum of $G^k$ and $G$ are not related}\label{seubsec:}
In the general situation when the spectrum of $G^k$ and $G$ are not related, one can make use of the following recent spectral bounds for $\alpha_k$ given in \cite{acf2019}. Let $G$ be a graph with eigenvalues $\lambda_1\ge \cdots\ge \lambda_n$. Let $[2,n]=\{2,3,\ldots,n\}$. Given a polynomial $p\in \Real_k[x]$, consider the following parameters:
\begin{itemize}
\item
$W(p) = \max_{u\in V}\{(p(A))_{uu}\}$,
\item
$w(p) = \min_{u\in V}\{(p(A))_{uu}\}$,
\item
$\Lambda(p) = \max_{i\in[2,n]}\{p(\lambda_i)\}$,
\item
$\lambda(p) = \min_{i\in[2,n]}\{p(\lambda_i)\}$.
\end{itemize}

\begin{theorem}{\rm (Abiad, Coutinho, Fiol  \cite{acf2019})}
\label{thmACF2019}.
Let $G$ be a graph with $n$ vertices and eigenvalues
$\lambda_1\ge\cdots \ge \lambda_n$.
\begin{itemize}
\item[$(i)$]
{\bf An inertial-type bound.}  Let  $p\in \Real_k[x]$ with corresponding parameters $W(p)$ and $\lambda(p)$. Then,
\begin{equation}
\label{eq:thm1}
\alpha_k\le \min\{|i : p(\lambda_i) \ge w(p)| , |i : p(\lambda_i) \le W(p)|\}.
\end{equation}
\item[$(ii)$] {\bf A ratio-type bound.}
Assume that $G$ is regular. Let  $p\in \Real_k[x]$ such that $p(\lambda_1) > \lambda(p)$. Then,
\begin{equation}
\label{eq:thm2}
\alpha_k\le n\frac{W(p)-\lambda(p)}{p(\lambda_1)-\lambda(p)}.
\end{equation}
\end{itemize}
\end{theorem}


In Section \ref{sec:newboundschik} we shall prove new eigenvalue lower bounds for $\chi_k$ which only require the use of the spectrum of $G$, this also useful when the spectrum of $G$ and $G^k$ are not related.

\subsection{Partially walk-regular graphs}\label{sec:partiallywalkregular}
A graph $G$ is called {\em $k$-partially walk-regular}, for some integer $k\ge 0$, if the number of closed walks of a given length $l\le k$, rooted at a vertex $v$, only depends on $l$. Thus, every (simple) graph is $k$-partially walk-regular for $k=0,1$, every regular graph is $2$-partially walk-regular and, more generally, every $k$-partially distance-regular is $2k$-partially walk-regular. Moreover $G$ is $k$-partially walk-regular for any $k$ if and only if $G$ is walk-regular, a concept introduced by Godsil and Mckay in \cite{gm80}.  For example, it is well-known that every distance-regular graph is walk-regular (but the converse does not hold).
In other words, if $G$ is $k$-partially walk-regular, for any polynomial $p\in \Real_k[x]$ the diagonal of $p(A)$ is constant with entries
$$
(p(A))_{uu}=w(p)=W(p)=\frac{1}{n}\tr p(A)=\frac{1}{n}\sum_{i=1}^n  p(\lambda_i)\quad \forall u\in V.
$$
Then, with $p\in \Real_k[x]$, \eqref{eq:thm1} and \eqref{eq:thm2} become
\begin{equation}
\label{eq:thm1-pwr}
\textstyle
\alpha_k\le \min\{|i : p(\lambda_i) \ge \frac{1}{n}\sum_{i=1}^n p(\lambda_i)| , |i : p(\lambda_i) \le \frac{1}{n}\sum_{i=1}^n  p_k(\lambda_i)|\}
\end{equation}
and
\begin{equation}
\label{general-bound-alpha_k}
\alpha_k\le \frac{\sum_{i=1}^n p(\lambda_i)-n\cdot\lambda(p)}{p(\lambda_1)-\lambda(p)}.
\end{equation}

In particular, notice that if $\tr p(A)=\sum_{i=1}^{n} p(\lambda_i) = 0$, inequality \eqref{general-bound-alpha_k} becomes
\\
\begin{equation}
\label{Hoffman-wr}
\alpha_k\leq \frac{n}{1-\frac{p(\lambda_1)}{\lambda(p)}}.
\end{equation}
\\
This can be seen a generalization of Hoffman bound \eqref{bound:hoffman}, since
it is obtained when, in \eqref{Hoffman-wr}, we take $k=1$ and $p_k(x)=x$ (in this case, note that  $p(\lambda_1)=\lambda_1$ and $\lambda(p)=p(\lambda_n)=\lambda_n$).
\\
In fact, for this case of partially $k$-walk-regular graphs, Fiol \cite{fiol20} proved that the upper bound in \eqref{general-bound-alpha_k} also applies for the Shannon capacity  $\Theta$ \cite{s56} and the Lov\'asz theta number $\vartheta$ \cite{l79} of $G^k$.

An alternative, and more direct proof of \eqref{Hoffman-wr} is the following. Let $G$ have adjacency matrix $A$, and let $U=\{1,2,\ldots,\alpha_k\}$ be a maximal $k$-independent set in $G$, such that the first vertices of $A$ correspond to $U$. Put $\vu=(x \1\, |\, \1)^\top$, where $x$ is a variable such that the values of $x$ correspond to the vertices in the maximal $k$-independent set $U$. Now consider the function
$$
\phi(x)=\frac{\langle \vu,p(A)\vu \rangle}{||\vu||^2}=\frac{2\alpha_k p(\lambda_1)x+(n-2\alpha_k)p(\lambda_1)}{\alpha_k x^2+n-\alpha_k},
$$
which attains a minimum at $x_{\min}=1-\frac{n}{\alpha_k}$. Thus, $\phi(x_{\min})$ gives
$$
\lambda(p)\leq \phi(x_{\min})=\frac{p(\lambda_1)}{1-\frac{n}{\alpha_k}},
$$
whence \eqref{Hoffman-wr} follows.
The same proof idea was used to extend the ratio bound for oriented hypergraphs, but using the normalized Laplacian spectrum \cite{amz2020}.

\subsubsection{Optimizing the upper bounds for $\alpha_k$}\label{sec:optimizingboundsalphak}
Notice that the bounds \eqref{eq:thm1-pwr} and \eqref{general-bound-alpha_k} are invariant under scaling and/or translating the polynomial $p$. Thus, when we are looking for the best polynomials, we can restrict ourselves to the following cases:
\begin{description}
\item[Bound \eqref{eq:thm1-pwr}:] upon changing the sign of an optimal solution $p$, we can always assume we are trying to find $p$ that minimizes $|\{i : p(\lambda_i) \geq w(p)\}|$. Moreover, a constant can be added to $p$ to make $w(p) = 0$. Thus, we get
\begin{equation}
\label{eq:thm1-pwr-simp}
\textstyle
\alpha_k\le \min\{|i : p(\lambda_i) \ge 0|\}.
\end{equation}
The optimization of this bound will be investigated in Section \ref{sec:directconsequenceourinertiabound}.
  \item[Bound \eqref{general-bound-alpha_k}:] we consider two simple possibilities:
  \begin{itemize}
\item[$(a)$]
If $p=f\in \Real[k]$ is a polynomial satisfying $\lambda(f)=0$ and $f(\theta_0)=1$,  the best result is obtained with the so-called {\em minor polynomial} $f_k$ that minimizes  $\sum_{i=0}^d m_i f_k(\theta_i)$. This case was studied by Fiol in \cite{fiol20}.
This polynomial can be found by solving the following linear programming problem (LP):
Let $f_k$ be defined by $f_k(\theta_0)=x_0=1$ and $f_k(\theta_i)=x_i$, for $i=1,\ldots,d$,  where the vector $(x_1,x_2,\ldots,x_d)$ is a solution of
\begin{equation}
\boxed{
\begin{array}{rl}
 {\tt minimize} & \sum_{i=0}^d m_ix_i\\
 {\tt subject\ to} & f[\theta_0,\ldots,\theta_m]=0,\ m=k+1,\ldots,d\\
                       & x_i\ge 0,\ i=1,\ldots,d\\
\end{array}
}
\label{LPminorpoly}
\end{equation}
Here, $f[\theta_0,\ldots,\theta_m]$ denote the $m$-th divided differences of Newton interpolation, recursively defined by  $f[\theta_i,\ldots,\theta_j]=\frac{f[\theta_{i+1},\ldots,\theta_j]-f[\theta_i,\ldots,\theta_{j-1}]}
{\theta_j-\theta_{i}}$, where $j>i$, starting with $f[\theta_i]=p(\theta_i)=x_i$, $0\le i\le d$.
Note that by equating these values to zero, we guarantee that $f_k\in \Real_k[x]$.
For more details about the minor polynomials, see  \cite{fiol20}.
Then, we get
\begin{equation}
\label{chi_k-minor-pol}
\alpha_k\le \sum_{i=0}^d m_i p_k(\theta_i)=\tr p_k(A),\qquad \mbox{and}\qquad   \chi_k\ge \frac{n}{ \sum_{i=0}^d m_i p_k(\theta_i)}.
\end{equation}
\item[$(b)$]
If $p=g\in \Real_k[x]$ is the polynomial satisfying $\sum_{i=0}^d m_i g(\theta_i)=0$ and $\lambda(g)=-1$, Eq. \eqref{Hoffman-wr}, with $\lambda_1=\theta_0$, gives
\begin{equation}
\label{chi_k-tr=0}
\alpha_k\le \frac{n}{1+g(\theta_0)},\qquad \mbox{and}\qquad \chi_k\ge 1+g(\theta_0).
\end{equation}
Hence, the best result is now obtained by maximizing $g(\theta_0)$. If $g(\theta_i)=x_i$ for $i=0,\ldots,d$, this leads to the following LPP:
\begin{equation}
\boxed{
\begin{array}{rl}
 {\tt maximize} &  x_0\\
 {\tt subject\ to} & \sum_{i=0}^d m_i x_i=0 \\
  & f[\theta_0,\ldots,\theta_m]=0,\ m=k+1,\ldots,d\\
                       & x_i=z_i-1,
                       z_i\ge 0,\ i=1,\ldots,d\\
\end{array}
}
\label{LLP}
\end{equation}
\end{itemize}

Consequently, both results $(a)$ and $(b)$ are equivalent in the sense that the best polynomial in $(a)$ yields the same results as the best polynomial in $(b)$. In the first case, $f_k$ is the polynomial that minimizes $\sum_{i=0}^d m_i f_k(\theta_i)$, subject to $f_k(\theta_i) \ge 0$ for any $i=1,\ldots,d$, and $f_k(\theta_0)=1$. In the second case, $g$ is the polynomial that maximizes $g(\lambda_0)$ under the conditions $g(\theta_i) \ge -1$ for any $i=1,\ldots,d$ and $\sum_{i=0}^d m_i g(\theta_i)=0$.
Now, suppose that $g$ satisfies the conditions in $(b)$. Then, then the polynomial $f_k=\frac{g+1}{g(\theta_0)+1}$ satisfies the conditions in $(a)$ and we get
$$
\textstyle
\alpha_k\le \sum_{i=0}^d m_i f_k(\theta_i)=\frac{1}{g(\theta_0)+1}\left[\sum_{i=0}^d m_i g(\theta_i)+n\right]=\frac{n}{1+g(\theta_0)},
$$
as expected. Similarly, if $f_k$ satisfies the conditions in $(a)$, then the polynomial
$
g=\frac{n f_k-\sum_{i=0}^d m_i f_k(\theta_i)}{\sum_{i=0}^d m_i f_k(\theta_i)}
$
satisfies the conditions in $(b)$, and yields the expected bound
$$
\textstyle
\alpha_k\le \frac{n}{1+g(\theta_0)}=\frac{1}{n}\sum_{i=0}^d m_i f_k(\theta_i).
$$
\end{description}

\section{New spectral bounds for $\chi_k$}\label{sec:newboundschik}

 In this section we prove several eigenvalue lower bounds for $\chi_k$ which only require the spectrum of $G$.

\subsection{First inertial-type bound for $\chi_k$}\label{sec:directconsequenceourinertiabound}


The first inertial-type bound is a consequence of the bound for $\alpha_k$ in \eqref{eq:thm1} (for a general value of $k$, an infinite class of graphs which attain such a bound is shown in \cite{act2016}):
\begin{equation}
\label{naiveirectlowerbound}
\chi_k(G) \ge \frac{n}{\min\{|i : p_k(\lambda_i) \ge w(p_k)| , |i : p_k(\lambda_i) \le W(p_k)|\}}.
\end{equation}

In the case of $k$-partially walk-regular graphs, the optimization of such bounds has already been discussed in Section \ref{sec:partiallywalkregular}.


We should note that if one considers $p_2(A)=A^2$ the bound (\ref{naiveirectlowerbound}) becomes:
\begin{equation}
\label{inertialikeboundk2distancechromaticnumber}
\chi_2(G) \ge \frac{n}{\min\{|i : \lambda_i^2 \ge \delta| , |i : \lambda_i^2 \le \Delta|\}},
\end{equation}
and this bound is tight for an infinite family of graphs. Indeed, consider
the incidence graph $G$ of a projective plane $PG(2, q)$, then $G^2$ has two cliques of size $q^2 + q + 1$ (corresponding to the points and lines, since any two points are incident to a common line and any two lines are incident to a common point). Therefore, $\chi(G^2) \geq q^2 + q + 1$. This is an example that Alon and Mohar use in \cite{am2002}. Note that (\ref{inertialikeboundk2distancechromaticnumber}) gives the same bound, as the eigenvalues of $G$ are $q+1, \sqrt{q}, 0, - \sqrt{q}$ and $-q-1$. In particular, $w_2(G) = W_2(G) = q + 1$ (the degree of the graph), whereas there are only two eigenvalues $q+1$ and $-q-1$ whose square is $\geq + 1$. So, as per the inertial-type bound from \cite{act2016}, $\alpha(G^2) \leq 2$, and hence $\chi(G^2) \geq \frac{2(q^2+q+1)}{2}$.

\subsubsection{Optimization of the first inertial-type bound}
\label{subsec:opt-first}

Our goal is to introduce a mixed integer linear  program (MILP) to compute the best polynomial giving the above bound in \eqref{eq:thm1} (and hence the same for the bound in \eqref{naiveirectlowerbound}). Since such a bound is also valid for the quantum $k$-independence number and this parameter is not computable in general, the use MILPs to find the best polynomial is justified.

Let $G$ have spectrum $\spec G=\{\theta_0^{m_0},\ldots, \theta_d^{m_d}\}$. Upon changing the sign of an optimal solution $p$, we can always assume we are trying to find $p\in \Real_k[x]$, and minimizing $|\{i : p(\lambda_i) \geq w(p)\}|$ or, in term of multiplicities, $\min \sum_{j: p(\theta_j)\geq w(p)} m_j$. Moreover, assuming that $w(p) = p(A)_{uu}$ for some vertex $u \in V(G)$, a constant can be added to $p(x)$ making $w(p) = 0$.

Let
$
p(x) = a_k x^k +\cdots + a_0,
$
$\b=(b_0,\ldots,b_d) \in \{0,1\}^{d+1}$, and $\m=(m_0,\ldots,m_d)$.
The following mixed integer linear program (MILP), with variables $a_0,\ldots,a_k$ and $b_1,\ldots, b_d$, finds the best polynomial for the bound \eqref{eq:thm1}:

\begin{center}
\begin{equation}
\boxed{
\begin{array}{rl}
{\tt minimize} & \m^\Tr \b\\
{\tt subject\ to} & \sum_{i = 0}^k a_i (A^i)_{vv} \geq 0,\quad  v \in V(G)\setminus \{u\}\\
 & \sum_{i = 0}^k a_i (A^i)_{uu} = 0\\
 & \sum_{i = 0}^k a_i \theta_j^{\ i} - M b_j + \epsilon \leq 0, \quad  j = 0,...,d\quad (\ast)\\
 & \b \in \{0,1\}^{d+1}
		\end{array}}
\label{MILP:inertia}
\end{equation}
\end{center}
Here $M$ is set to be a large number, and $\epsilon$ small.
The idea of this formulation is that each $b_j = 1$ represents an index $j$ so that $p(\theta_j) \geq w(p) = 0$. In fact, condition $(\ast)$ gives that $p(\theta_j)\ge 0$ implies $b_j=1$.
 So, upon minimizing the quantity of such indices $j$, we are optimizing $p(x)$ and the corresponding bound  $\alpha_k\le \m^\Tr \b$. For each $u \in V(G)$, we write one such MILP and find the best objective value of all.
With respect to the choices for $\epsilon$ and $M$, note that we can always set $\epsilon = 1$ as scaling of the $a_i$'s is allowed. If the $M$ chosen is not large enough, the MILP will be unfeasible and we can repeat with a larger $M$.


In Table \ref{tableMILP}, the results of the MILP optimal bound \eqref{MILP:inertia} are shown for all named graphs in Sage with less than 100 vertices and diameter at least 3. We compare these to the Lov\'asz theta number of $G^k$ and the exact value of $\alpha_2$. For regular graphs, the bound from Corollary 3.3 in \cite{acf2019} is also included. Observe that the bound in \cite{acf2019} generally outperforms our MILP for the graphs in Table \ref{tableMILP}. However, it should be noted that this bound requires regularity, whereas the MILP bound \eqref{MILP:inertia} is also applicable to irregular graphs. Table \ref{tableprop} shows for $n = 4,\dots 9$ the proportion of irregular graphs on $n$ vertices for which the optimal solution of our MILP matches the actual value of $\alpha_2$.

\begin{table}
	\scriptsize{
		\begin{center}
			\begin{tabular}{|lccccc|}
				\hline
				Name & Bound in \cite{acf2019} & $\vartheta_2$ \cite{l79} & Inertial-type bound MILP \eqref{MILP:inertia} & Inertial-type bound MILP \eqref{MILP:inertia2} & $\alpha_2$\\
				\hline
				Balaban 10-cage & $17$ & $17$ & $19$ & $19$ & $17$ \\
				Frucht graph & $3$ & $3$ & $3$ & $3$ & $3$ \\
				Meredith Graph & $14$ & $10$ & $10$ & $10$ & $10$ \\
				Moebius-Kantor Graph & $4$ & $4$ & $6$ & $4$ & $4$ \\
				Bidiakis cube & $3$ & $2$ & $4$ & $3$ & $2$ \\
				Gosset Graph & $2$ & $2$ & $8$ & $2$ & $2$ \\
				Gray graph & $14$ & $11$ & $19$ & $19$ & $11$ \\
				Nauru Graph & $6$ & $5$ & $8$ & $8$ & $6$ \\
				Blanusa First Snark Graph & $4$ & $4$ & $4$ & $4$ & $4$ \\
				Pappus Graph & $4$ & $3$ & $7$ & $6$ & $3$ \\
				Blanusa Second Snark Graph & $4$ & $4$ & $4$ & $4$ & $4$ \\
				Poussin Graph & - & $2$ & $4$ & - & $2$ \\
				Brinkmann graph & $4$ & $3$ & $6$ & $6$ & $3$ \\
				Harborth Graph & $12$ & $9$ & $13$ & $13$ & $10$ \\
				Perkel Graph & $10$ & $5$ & $18$ & $18$ & $5$ \\
				Harries Graph & $17$ & $17$ & $18$ & $18$ &$17$ \\
				Bucky Ball & $16$ & $12$ & $16$ & $16$ & $12$ \\
				Harries-Wong graph & $17$ & $17$ & $18$ & $18$ & $17$ \\
				Robertson Graph & $3$ & $3$ & $5$ & $5$ & $3$ \\
				Heawood graph & $3$ & $2$ & $2$ & $3$ & $2$ \\
				Herschel graph & - & $2$ & $3$ & - & $2$ \\
				Hoffman Graph & $3$ & $2$ & $5$ & $4$ & $2$ \\
				Sousselier Graph & - & $3$ & $5$ & - & $3$ \\
				Sylvester Graph & $6$ & $6$ & $10$ & $10$ & $6$ \\
				Coxeter Graph & $7$ & $7$ & $7$ & $7$ & $7$ \\
				Holt graph & $6$ & $3$ & $7$ & $8$ & $3$ \\
				Szekeres Snark Graph & $12$ & $10$ & $13$ & $14$ & $9$ \\
				Desargues Graph & $5$ & $5$ & $6$ & $6$ & $4$ \\
				Horton Graph & $26$ & $24$ & $30$ & $30$ & $24$ \\
				Kittell Graph & - & $3$ & $5$ & - & $3$ \\
				Tietze Graph & $3$ & $3$ & $4$ & $3$ & $3$ \\
				Double star snark & $7$ & $7$ & $9$ & $10$ & $6$ \\
				Krackhardt Kite Graph & - & $2$ & $4$ & - & $2$ \\
				Durer graph & $3$ & $2$ & $3$ & $3$ & $2$ \\
				Klein 3-regular Graph & $13$ & $13$ & $19$ & $19$ & $12$ \\
				Truncated Tetrahedron & $3$ & $3$ & $4$ & $4$ & $3$ \\
				Dyck graph & $8$ & $8$ & $8$ & $8$ & $8$ \\
				Klein 7-regular Graph & $3$ & $3$ & $9$ & $3$ & $3$ \\
				Ellingham-Horton 54-graph & $14$ & $12$ & $20$ & $20$ & $11$ \\
				Tutte-Coxeter graph & $8$ & $6$ & $10$ & $10$ & $6$ \\
				Ellingham-Horton 78-graph & $21$ & $19$ & $27$ & $27$ & $18$ \\
				Tutte Graph & $11$ & $10$ & $13$ & $13$ & $10$ \\
				Errera graph & - & $2$ & $4$ & - & $2$ \\
				F26A Graph & $6$ & $6$ & $7$ & $7$ & $6$ \\
				Watkins Snark Graph & $14$ & $9$ & $13$ & $13$ & $9$ \\
				Flower Snark & $5$ & $5$ & $7$ & $7$ & $5$ \\
				Markstroem Graph & $6$ & $6$ & $7$ & $7$ & $6$ \\
				Wells graph & $6$ & $3$ & $9$ & $10$ & $2$ \\
				Folkman Graph & $4$ & $3$ & $5$ & $5$ & $3$ \\
				Wiener-Araya Graph & - & $8$ & $12$ & - & $8$ \\
				Foster Graph & $22$ & $22$ & $23$ & $23$ & $21$ \\
				McGee graph & $6$ & $5$ & $7$ & $6$ & $5$ \\
				Franklin graph & $3$ & $2$ & $4$ & $3$ & $2$ \\
				Hexahedron & $2$ & $2$ & $2$ & $2$ & $2$ \\
				Dodecahedron & $5$ & $4$ & $4$ & $4$ & $4$\\
				Icosahedron & $2$ & $2$ & $4$ & $2$ & $2$ \\
				\hline
			\end{tabular}
	\end{center}}
	\caption{Comparison between different bounds for $\alpha_2$.}
	\label{tableMILP}
\end{table}
\normalsize

\begin{table}
		\begin{center}
			\begin{tabular}{|l|c|c|c|c|c|c|}
				\hline
				Number of vertices & 4 & 5 & 6 & 7 & 8 & 9\\
				\hline
				Proportion & 0.86 & 0.84 & 0.76 & 0.62 & 0.46 & 0.27\\
				\hline
			\end{tabular}
		\end{center}
		\caption{Proportion of small graphs for which the optimal value of the MILP coincides with $\alpha_2$.}
		\label{tableprop}
\end{table}

In the case of $k$-partially walk-regular graphs, we only need to run the MILP \eqref{MILP:inertia} once, since all vertices have the same number of closed walks of length smaller of equal than $k$. Then, the problem can be formulated follows:

Let $G$ be a $k$-partially walk-regular graph with diameter $D$ and $\spec G=\{\theta_0^{m_0},\theta_1^{m_1},\ldots,\theta_d^{m_d}\}$. For a given $k<D\ (\le d)$, let
$ p(x) = a_k x^k +\cdots +a_0$,
$\b=(b_0,\ldots,b_d)\in \{0,1\}^{d+1}$ and $\m=(m_0,\ldots,m_d)$.
Now, the following MILP \eqref{MILP:inertiaWR}, with variables $a_1,\ldots,a_k$ and $b_0,\ldots, b_d$, finds the best polynomial and the corresponding bound for $\alpha_k$:

\begin{center}
	\begin{equation}
        \boxed{
		\begin{array}{rl}
		 &	\\
		{\tt minimize} & \m^\Tr \b\\
        {\tt subject\ to} & \sum_{i = 0}^d m_i p(\theta_i)= 0\\
		& \sum_{i = 0}^k a_i \theta_j^{\ i} - Mb_j + \epsilon \leq 0, \quad j = 0,...,d \\
		& \b \in \{0,1\}^{d+1}
		\end{array}}
        \label{MILP:inertiaWR}
    \end{equation}
\end{center}

Observe that the target polynomial $p$ in \eqref{MILP:inertiaWR} could be written as a linear combination of the predistance polynomials $p_1,\ldots,p_k$, since all of them are orthogonal to $p_0=1$ with respect to the scalar product in \eqref{escalar-prod}: $\langle p_j, 1\rangle_G=\frac{1}{n}\tr p_j(A)=w(p_j)=0$, $j=1,\ldots,k$, and, hence, so is $p$. This allows us to remove the first constraint in \eqref{MILP:inertia}.

Next we illustrate how the MILP \eqref{MILP:inertiaWR} can be used to find the best polynomials to upper bound $\alpha_k$ for an infinite family of Odd graphs. For every integer $\ell\geq 2$, the Odd graphs $O_{\ell}$ constitute a well-known family of distance-regular graphs with interactions between graph theory and other areas of combinatorics, such as coding theory and design theory. The vertices of $O_{\ell}$ correspond to the $\ell -1$ subsets of a $(2\ell -1)$-set, and adjacency is defined by void intersection. Note that $O_3$ is the Petersen graph. In general, $O_{\ell}$  is an $\ell$-regular graph of order $n=\binom{2\ell -1}{\ell -1}=\frac{1}{2}\binom{2\ell}{\ell}$, diameter $D=\ell-1$, and its eigenvalues and multiplicities are $\theta_i=(-1)^i(\ell-i)$ and $m(\theta_i)=m_i=\binom{2\ell-1}{i}-\binom{2\ell-1}{i-1}$ for $i=0,1,\ldots, \ell-1$.

For the case $k=D-1=\ell-2$, where $\alpha_k$ is the maximum number of vertices mutually at distance $D$, we have the following result:
\begin{proposition}
For the Odd graph $O_{\ell}$, with diameter $D=\ell-1$, the $(D-1)$-independence number $\alpha_{D-1}=\alpha_{\ell-2}$ satisfies the bound
\begin{equation}
\label{propo:alphaOl}
\alpha_{\ell-2}(O_{\ell})\le
\left\{
\begin{array}{ll}
2\ell-2 & \mbox{for odd $\ell$,} \\
2\ell-1 & \mbox{for even $\ell$.}
\end{array}
\right.
\end{equation}
\end{proposition}
\begin{proof}
We claim that, for such graphs, the polynomial $p\in \Real_{\ell-2}[x]$ obtained from the MILP problem has zeros $z_i$ for $i=2,\ldots,\ell-1$, where $z_{i}=\theta_{i}+(-1)^{\lceil\frac{i+1}{2} \rceil}\varepsilon$ for odd $i$,
$z_{i}=\theta_{i}+(-1)^{\lfloor\frac{i-1}{2} \rfloor}\varepsilon$ for even $i$, and $\varepsilon$ is the solution in $(0,1)$ of the equation
\begin{equation}
\label{tr(p)=0}
\phi(\varepsilon)=\sum_{i=0}^{\ell} m_i p(\theta_i)=0.
\end{equation}
The reason is that this polynomial satisfies the main condition \eqref{tr(p)=0} of the MILP problem,
and ($p$ or $-p$) minimizes  the number of 1's in the vector $\b$. More precisely, from the definition of $p$ it is readily checked that
\begin{itemize}
\item
If $\ell$ is odd, then $-p(\theta_1)>0$ and $-p(\theta_i)<0$ for $i=0,2,\ldots,\ell$.
\item
If $\ell$ is even, then $p(\theta_i)>0$ for $i=0,1$, and $p(\theta_i)<0$ for $i=2,\ldots,\ell$.
\end{itemize}
In other words, in the first case $\b=(0,1,0,\ldots,0)$, and hence, $\alpha_{\ell-2}\le m_1=2\ell-2$;
whereas, in the second case, $\b=(1,1,0,\ldots,0)$, and hence, $\alpha_{\ell-2}\le m_0+m_1=2\ell-1$,
as claimed.
\end{proof}

In Table \ref{tableOddgraph} we show some examples of the results obtained for $\ell=4,\ldots,8,10,12,14$.
For the first values, we also indicate the polynomial $\phi(\varepsilon)$, which  is shown to be monic with a convenient scaling (obtained dividing \eqref{tr(p)=0} by $\pm {2\ell-1\choose \ell-1}$), together with its ``key zero" $\varepsilon_0\in (0,1)$.
Also, we compare the obtained MILP bound  with the exact value of $\alpha_k$.

\begin{table}[H]
	\centering
\begin{tabular}{ |l|l|l|c| }
 \hline
\multirow{4}{*}{$\alpha_2(O_{4})$} &  Bound from the MILP & $7=m_0+m_1$ \\
  & Polynomial $\varepsilon^2+3\varepsilon-2$  & 0.561552813 \\
  & Exact value $\alpha_2$ & 7 \\
\hline
\multirow{4}{*}{$\alpha_3(O_{5})$} & Bound from the MILP & $8=m_1$ \\
  & Polynomial $\varepsilon^3-12\varepsilon+4$  &  0.336508805 \\
  & Exact value $\alpha_3$ & 7 \\
\hline
\multirow{3}{*}{$\alpha_4(O_{6})$} & Bound from the MILP & $11=m_0+m_1$ \\
 & Polynomial $\varepsilon^4+4\varepsilon^3 -46\varepsilon+12$ & 0.238605627 \\
 & Exact value $\alpha_4$ & 11 \\
\hline	
\multirow{3}{*}{$\alpha_5(O_{7})$} & Bound from the MILP & $12=m_1$ \\
 & Polynomial $\varepsilon^5-\varepsilon^4-41\varepsilon^3+41\varepsilon^2+246\varepsilon-36$ &   0.1434068868 \\
 & Exact value $\alpha_5$ & 12 \\
\hline
\multirow{3}{*}{$\alpha_6(O_{8})$} & Bound from the MILP & $15=m_0+m_1$ \\
 & Polynomial $\varepsilon^6+7\varepsilon^5-45\varepsilon^4-287\varepsilon^3+256\varepsilon^2+1372\varepsilon-144$  &  0.1032025452 \\
 & Exact value $\alpha_6$ & 15 \\
\hline
\hline
$\alpha_8(O_{10})$ & Bound from the MILP & $19=m_0+m_1$ \\
  &  Exact value $\alpha_b$ & 19 \\
\hline
$\alpha_{10}(O_{12})$ & Bound from the MILP & $23=m_0+m_1$ \\
 &  Exact value $\alpha_{10}$ & 23 \\
\hline
$\alpha_{12}(O_{14})$ & Bound from the MILP & $27=m_0+m_1$ \\
 &  Exact value $\alpha_{12}$ & 27 \\
\hline
\end{tabular}
 \caption{Infinite family of Odd graphs for which the output from MILP \eqref{MILP:inertiaWR} gives the best polynomials for upper bounding $\alpha_k$.} \label{tableOddgraph}
 \end{table}

Note that, when $\ell$ increases, $\varepsilon$ tends to zero and hence the target polynomial $p$ is closer and closer to the minor polynomial $f_k$ up to a constant multiplicative factor.
This gives an interesting view of the relationship between the inertial- and ratio-type methods. Moreover, the same result of Proposition \ref{propo:alphaOl} can also be proved by using only the minor polynomials, see \cite{fiol20}.
Also, notice that, except for the Odd graph $O_5$, all the obtained bounds are tight.
In fact,  in the even case $\ell=2k$, one can check that the vertices at maximum distance $2k-1$ from each other constitute a $2-(4k-1,2k-1,k-1)$ symmetric design (see \cite{h86} for its definition). Such combinatorial structures exist, at least, for $k=2,\ldots,7$  \cite{glasgow}, which give the optimal values in Table \ref{tableOddgraph} when $\ell=2,4,\ldots,14$. In particular, the 7 vertices of $O_4$  correspond to the lines (or the points) of the Fano plane (see Figure \ref{fig:fano}), and the 11 vertices of $O_6$ are the points of the Payley biplane.

\begin{figure}
	\begin{center}
		\includegraphics[scale=0.5]{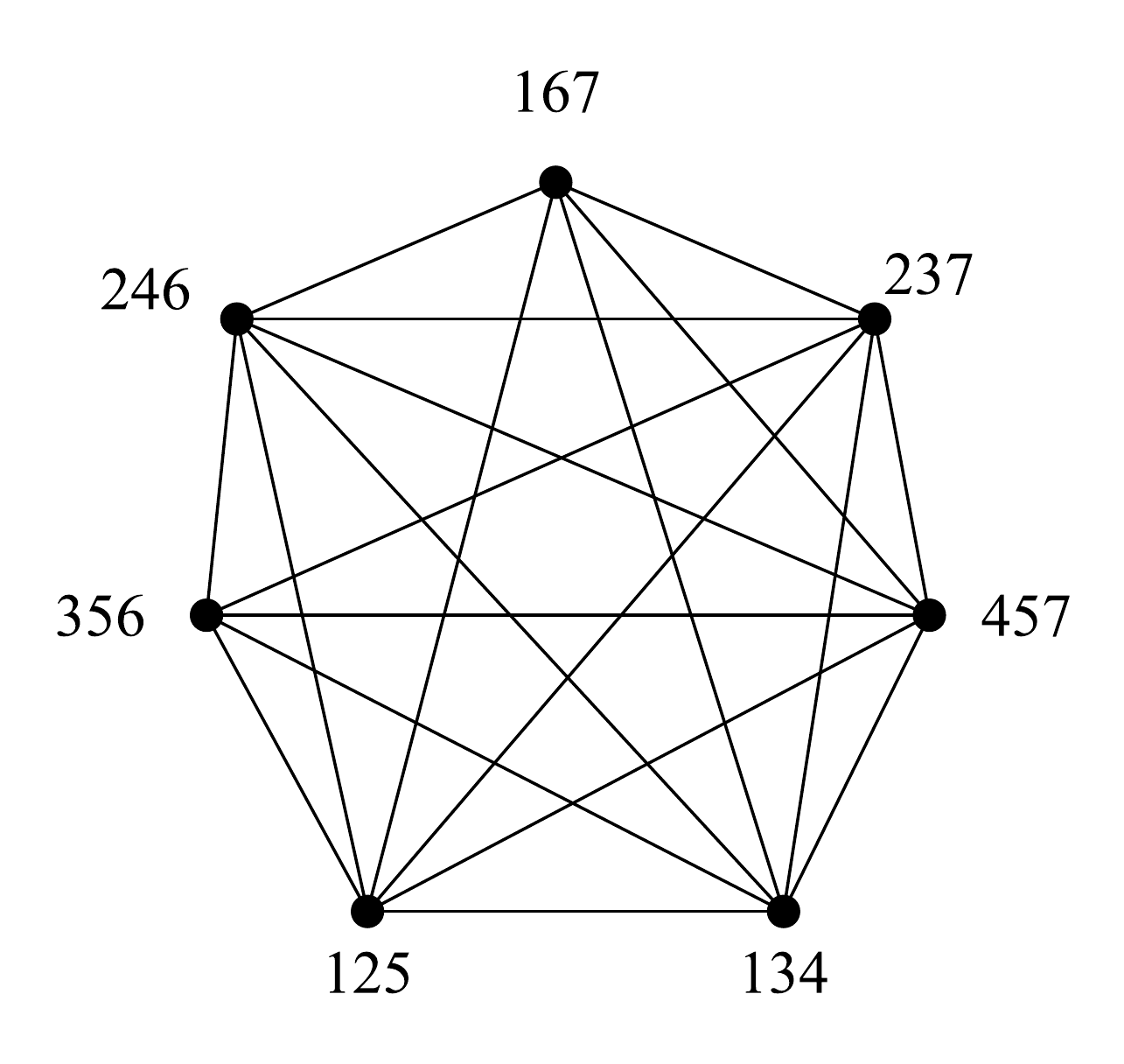}
	\end{center}
	\caption{Maximal set of vertices in $O_4$ ($\alpha_2=7$) at mutual distance 3 (each pair of vertices---lines of the Fano plane $F_7$---has exactly one common digit---vertex of $F_7$).}
\label{fig:fano}
\end{figure}

Another infinite family of graphs for which \eqref{MILP:inertiaWR} behaves nicely is a particular family of Cayley graphs. Let $G$ be a finite group with identity element $\mathbf{1}$ and let $S\subseteq G$. The (directed) Cayley graph $\Gamma(G,S)$ is a graph with vertex set $G$ and an arc for every pair $u,v\in G$ such that $uv^{-1}\in S$. If $S$ is inverse-closed and does not contain $\mathbf{1}$, then $\Gamma(G,S)$ is symmetric and loopless, in which case we may view it as a simple undirected graph. Consider for each $n \ge 3 $ the Cayley graph $\Gamma_n := \Gamma(D_{2n},S_{2n})$ on the dihedral group $D_{2n} = \langle a,b \mid a^n = b^2 = (ab)^2 = \mathbf{1}\rangle$ and inverse-closed subset $S_{2n} = \{a,a^{-1},b\} \subset D_{2n}$. Then $\{\Gamma_n\}_{n\ge 3}$ is a family of connected, $3$-regular graphs on $2n$ vertices.
The graph $\Gamma_n$ is known as the {\em prism} graph \cite{g87} and the above construction as a Cayley graph is due to Biggs \cite[pag. 126]{biggs}.  These graphs are vertex-transitive and, hence,  walk-regular, but not distance-regular. Thus the Delsarte LP bound does not apply. Table \ref{tableCayley} shows the behaviour of the MILP bound on $\Gamma_n$ for $3\le n \le 16$. Note that the optimal value equals exactly $\alpha_2$ when $n \neq 2 \mod 4$. This trend continues if we solve the MILP for larger values of $n$.
An easy way to prove that the exact values of $\alpha_2$ are those expected from the table ($\alpha_2=2k$ if $n=4k+i$ for $i=0,1,2$, and $\alpha_2=2k+1$ if $n=4k+3$) is to view $\Gamma_n$ as the Cayley graph on the Abelian group $\Z_n\times \Z_2$ with generating set $S=\{\pm(1,0),\pm(0,1)\}$. Then the graph can be represented by a plane tessellation with rectangles $n\times 2$ \cite{yfma85} (or embedding on the torus) which allows us a neat identification of the maximum $2$-independent vertex sets.

\begin{table}[h]
\begin{center}
\begin{tabular}{|l|c|c|c|c|c|c|c|c|c|c|c|c|c|c|}
\hline
$n$ & 3&4 & 5 & 6 & 7 & 8 & 9 & 10 & 11 & 12 & 13 & 14 & 15 & 16\\
\hline
MILP bound & 1 & 2 & 2 & 4 & 3 & 4 & 4 & 6 & 5 & 6 & 6 & 8 & 7 & 8\\
$\alpha_2$ & 1 & 2 & 2 & 2 & 3 & 4 & 4 & 4 & 5 & 6 & 6 & 6 & 7 & 8\\
\hline
\end{tabular}
\end{center}
\caption{An infinite family of Cayley graphs $\Gamma_n$ for which the MILP bound equals $\alpha_2$ when $n \neq 2 \mod 4$.}
\label{tableCayley}
\end{table}

\subsection{Second inertial-type bound for $\chi_k$}
\label{sec:extendingElphickWocjanforchik}

The bound (\ref{naiveirectlowerbound}) can be strengthened when $k = 1$ and $p_k(A) = A$ as follows (see Elphick and Wokjan \cite[Th. 1]{elphick}). Let $n^+=|i:\lambda_i>0|$, $n^0=|i:\lambda_i=0|$, and $n^-=|i:\lambda_i<0|$. Then,
\begin{equation}
\label{bound-EW}
\chi(G) \ge 1 + \max \left(\frac{n^+}{n^-} , \frac{n^-}{n^+}\right) \ge \frac{n}{n^0 + \min\{n^+ , n^-\}},
\end{equation}
with equality for the two bounds only if $n^0 = 0$, since
\[
1 + \max \left(\frac{n^+}{n^-} , \frac{n^-}{n^+}\right) =
\frac{n^+  + n^-}{\min\{n^+ ,  n^-\}}.
\]
The goal of this section is to extend the inertial-type bound \eqref{bound-EW} to the distance chromatic number $\chi_k(G)$ in the case when $G$ is $k$-partially walk-regular.
\begin{theorem}	\label{theo:extended-EW}
Let $G$ be a $k$-partially walk-regular graph with adjacency matrix eigenvalues $\lambda_1\geq \cdots \geq \lambda_n$. Let $p_k\in \mathbb{R}_k[x]$ such that $\sum_{i=1}^{n} p_k(\lambda_i) = 0$. Then,
\begin{equation}\label{extendedElphickWocjan}
	\chi_k\geq 1+\max{ \left(\frac{|j:p_k(\lambda_j)< 0|}{|j:p_k(\lambda_j)> 0|}\right)}.
\end{equation}
\end{theorem}

\begin{proof}
An analogous argument as it is used in \cite[Th. 1]{elphick} applies here by using $p_k(A)$ instead of $A$. The proof in \cite{elphick} relies on the fact that there exist $\chi$ unitary matrices $U_i$ such that:
\[
\sum_{i=1}^{\chi-1} U_iAU_i^* = -A.
\]
Now we consider $p_k(A)$ instead of $A$ and a $k$-partially walk-regular graph $G$ with $\sum_{i=1}^{n} p_k(\lambda_i) = 0$ (recall that $p_k(A)$ has constant zero diagonal if and only if $\text{tr}p_k (A) = 0$, or equivalently, $\sum_{i=1}^{n} p_k(\lambda_i) = 0$)). Then it follows that
\begin{equation}
\label{toolElphickWocjanextended}
\sum_{i=1}^{\chi_k-1} U_ip_k(A)U_i^* = -p_k(A).
\end{equation}
Observe that the above holds because Theorem 6 in \cite{ew2013} is also valid for weighted adjacency matrices with zero diagonal.
Let $v_1,\ldots,v_n$ be the eigenvectors of unit length corresponding to the eigenvalues $p_k(\lambda_1)\geq \cdots \geq p_k(\lambda_n)$. Let $p_k(A)=p_k(B)-p_k(C)$, where
$$
p_k(B)=\displaystyle\sum_{i=1}^{|j:p_k(\lambda_j)> 0|}p_k(\lambda_i)v_iv_i^{\ast},\qquad p_k(C)=\displaystyle\sum_{i=n-|j:p_k(\lambda_j)< 0|+1}^{n}-p_k(\lambda_i)v_iv_i^{\ast}.
$$
Observe that $p_k(B)$ and $p_k(C)$ are positive semidefinite matrices, and we also know that $\text{rank}(p_k(B))=|j:p_k(\lambda_j)> 0|$ and $\text{rank}(p_k(C))=|j:p_k(\lambda_j)< 0|$. Denote by $P^{+}$ and $P^{-}$ the orthogonal projectors onto the subspaces spanned by the eigenvectors corresponding to the positive and negative eigenvalues
of $p_k(A)$, respectively:
$$
P^{+}=\sum_{i=1}^{\text{rank}(p_k(B))}v_iv_i^{\ast} \qquad \text{and}\qquad P^{-}=\sum_{i=n-\text{rank}(p_k(C))+1}^{n}v_iv_i^{\ast}.
$$
Note that
$$
p_k(B)=P^{+}p_k(A)P^{+}\qquad \text{and}\qquad p_k(C)=-P^{-}p_k(A)P^{-}.
$$
Then, equation (\ref{toolElphickWocjanextended}) can be rewritten as follows
$$
\displaystyle \sum_{i=1}^{\chi_k-1}U_ip_k(B)U_i^{\ast}-\displaystyle \sum_{i=1}^{\chi_k-1}U_ip_k(C)U_i^{\ast}=p_k(C)-p_k(B),
$$
and, if we multiply both sides by $P^{-}$, we obtain
$$
P^{-}\displaystyle \sum_{i=1}^{\chi_k-1}U_ip_k(B)U_i^{\ast}P^{-}-P^{-}\displaystyle \sum_{i=1}^{\chi_k-1}U_ip_k(C)U_i^{\ast}P^{-}=p_k(C).
$$
Now, since we know that $P^{-}\displaystyle \sum_{i=1}^{\chi_k-1}U_ip_k(C)U_i^{\ast}P^{-}$ is positive semidefinite, we obtain
$$
P^{-}\displaystyle \sum_{i=1}^{\chi_k-1}U_ip_k(B)U_i^{\ast}P^{-} \succeq  p_k(C)
$$
(where, with $X,Y$ being matrices, $X\succeq Y$ means that $X-Y$ is positive semidefinite).
Finally, using the fact that the rank of a sum is less or equal than the sum of the ranks of the summands, that the rank of a product is less than or equal to the minimum of the ranks of the factors, and Lemma 2 in \cite{elphick} (if $X,Y\in \C^{n\times n}$ are positive semidefinite and  $X\succeq Y$, then $\rank(X)\ge \rank(Y)$), we obtain the desired inequality
$$
\left(\chi_k-1\right)|j:p(\lambda_j)> 0|\geq |j:p(\lambda_j)< 0|.
$$
\end{proof}

Note that the bound from Theorem \ref{theo:extended-EW} is equivalent to
\begin{equation*}
\label{extendedElphickWocjan2}
\chi_k\geq 1+\max{ \left(\frac{|j:p_k(\lambda_j)> 0|}{|j:p_k(\lambda_j)< 0|},\frac{|j:p_k(\lambda_j)< 0|}{|j:p_k(\lambda_j)> 0|}\right)}.
\end{equation*}
Observe also that the maximum  is taken over all polynomials $p_k$.

Regarding the second inertial-type bound \eqref{extendedElphickWocjan}, we note that not all graphs allow for an improvement of such bound due to the presence of zeros, and in that case one can better use the inertial-type bound \eqref{eq:thm1} which we optimize for $\alpha_k$ (and hence also for $\chi_k$) in Section \ref{subsec:opt-first}.

\subsubsection{Optimization of the second inertial-type bound}\label{optsecondinertialbound}

Similarly to our discussion in Section \ref{subsec:opt-first} for the optimization of the first inertial-type bound, we can use MILPs to optimize the polynomials appearing in the second inertial-type bound \eqref{extendedElphickWocjan}. For this bound, however, we must solve $n$ MILPs to obtain the best possible bound.
The procedure goes as follows: for each $\ell \in \{1,...,n-1\}$, we solve the following MILP:

\begin{center}
\begin{equation}
\boxed{
\begin{array}{rl}
{\tt maximize} & 1 + \frac{n-\1^\Tr \b}{\ell} \\
{\tt subject\ to} & \sum_{j=1}^n \sum_{i = 0}^k a_i \lambda_j^i = 0 \\
 & \sum_{i = 0}^k a_i \lambda_j^i - M b_j+\epsilon \leq 0, \quad j=1,...,n\\
 & \sum_{i = 0}^k a_i \lambda_j^i - M c_j \leq 0, \quad\qquad j=1,...,n\\
 & \sum_{i = 1}^n c_i = \ell \\
 & \b \in \{0,1\}^{n}, \quad \c \in \{0,1\}^{n}
\end{array}
}
\label{MILP:inertia2}
\end{equation}
\end{center}

Unlike the previous MILP \eqref{MILP:inertia}, which optimized the first inertial-type bound for $\chi_k$, for the above MILP we require $\b \in \{0,1\}^{n},\ \c \in \{0,1\}^{n}$ since here we look at all eigenvalues, including the repeated ones.
As before, the $a_i$ are the coefficients of the polynomial of degree at most $k$, say $p(x) = a_k x^k +\cdots + a_0$, and the first restriction is the hypothesis of the theorem, that is  $\tr p(A) = 0$. The second restriction implies that if $p(\lambda_j) \ge 0$, then $b_j = 1$, whereas the third restriction implies that if $p(\lambda_j) > 0$, then $c_j = 1$.
Thus, we have:
\begin{itemize}
\item
$|j:p_k(\lambda_j)> 0|=\1^\Tr\c=\sum_{i = 1}^n c_i = \ell$ (fourth restriction),
\item
$|j:p_k(\lambda_j)= 0|=\1^\Tr(\b-\c)$, and
\item
$|j:p_k(\lambda_j)< 0|=n-\1^\Tr \b$,
\end{itemize}
from where we set the function to maximize.

In theory, this MILP is a sound way to approximate Theorem \ref{theo:extended-EW}. However, in practice the limited precision of MILP solvers leads to implementation problems for certain graphs. Consider for example the prism graph $\Gamma_4$, for which MILP \eqref{MILP:inertia} was tight. Solving MILP \eqref{MILP:inertia2} with Gurobi for $k=2$, we find optimal value 7. This is clearly not a valid lower bound, as $\chi_2(\Gamma_4) = 4$. The corresponding optimal parameters are $p_2(x)=833.3249999999999x^2-1666.6499999999999x-2499.975$ and $b = (1,0,0,0,0,0,0,1)$, $c=(0,0,0,0,0,0,0,1)$. In other words, eigenvalue 3 is supposedly a root of $p_2$ and the other eigenvalues are not. However, due to rounding this is not exactly true: 3 is not a root of $p_2$, but it is a root of the polynomial $833 \frac{1}{3} x^2-1666\frac{2}{3}x-2500$ (or its monic equivalent $x^2-2x+3$), which has eigenvalue $-1$ as a second root. Eigenvalue 1 with multiplicity 3 is then the only eigenvalue such that this polynomial is negative, so the bound becomes $1+\frac{3}{1}=4$, which is tight.

In general, it is hard to prevent these types of errors, as no MILP solver has perfect accuracy. For $k=2$, we will consider a restriction of MILP \eqref{MILP:inertia2}, where this can be detected and prevented. For a regular graph $G$ with eigenvalues $d=\lambda_1\ge \lambda_2 \ge \dots \ge \lambda_n$, consider the polynomial $p_2(x)=x^2+bx-d$. This polynomial has two distinct roots $x_1<0<x_2$ such that $x_1x_2=-d$ and $b=-(x_1+x_2)$. Moreover, note that for any choice of $b$, it satisfies the trace condition $\sum_{i=1}^n p_2(\lambda_i)=0$. Therefore, it corresponds to a valid solution of MILP \eqref{MILP:inertia2}. Since the optimal polynomial is fixed up to the coefficient $b$, we can now calculate which eigenvalues are root pairs of $p_2$ and fix the bound accordingly.

To find an optimal value of $b$ we do not need to solve an MILP. Instead, the following greedy strategy suffices. Suppose $\lambda$ is the smallest negative eigenvalue such that $p_2(\lambda)<0$. To maximize the numerator of Equation \eqref{extendedElphickWocjan}, it is better to choose $x_1$ close to $\lambda$, as this will increase the value of $x_2$. For every negative eigenvalue $\lambda$, we therefore compute the bound for $x_1 = \lambda-\varepsilon$ with $\varepsilon>0$ small. By placing $x_1$ or $x_2$ close to 0, we also cover the cases where exactly all negative or all positive eigenvalues lie in the negative range of $p_2$. Finally, we set every eigenvalue as a root of $p_2$ and compute the corresponding lower bound. Observe that this strategy can easily be adapted for the polynomial $-p_2$, which also satisfies the trace condition. To obtain the best value bound, we consider all above cases for $p_2$ and $-p_2$ and take the maximum.

In Table \ref{tableMILP}, we compute the corresponding upper bound on $\alpha_2$ for the named Sage graphs and compare it to previous results. Note that these values are an upper bound for the actual optimum of MILP \eqref{MILP:inertia2}, as we restricted the optimal polynomial. On this particular set of graphs, the bound generally performs better than MILP \eqref{MILP:inertia}, most notably on the Gosset graph and Klein 7-regular graph. Like MILP \eqref{MILP:inertia}, MILP \eqref{MILP:inertia2} is tight for the incidence graphs of projective planes $PG(2,q)$ with $q$ a prime power and the prism graphs $\Gamma_n$ with $n\neq 2\mod 4$. Note that the latter are generalized Petersen graphs with parameters $(n,1)$. The bound is also tight for (generalised) Petersen graphs with $(n,k)\in \{(5,2),(8,3),(10,2)\}$. The second graph is also known as the M\"obius-Kantor graph and is walk-regular, but not distance-regular.

\subsection{First ratio-type bound for $\chi_k$}

Let $G$ be a graph with eigenvalues $\lambda_1\ge \cdots\ge \lambda_n$ and let $[2,n]=\{2,3,\ldots,n\}$. Given a polynomial $p_k\in \Real_k[x]$, recall the following parameters:
$W(p_k) = \max_{u\in V}\{(p_k(A))_{uu}\}$,
$w(p_k) = \min_{u\in V}\{(p_k(A))_{uu}\}$,
$\Lambda(p_k) = \max_{i\in[2,n]}\{p_k(\lambda_i)\}$,
$\lambda(p_k) = \min_{i\in[2,n]}\{p_k(\lambda_i)\}$.

Then notice that, for a regular graph, the upper bound (\ref{eq:thm2}) for $\alpha_k$  of Theorem \ref{thmACF2019}$(ii)$\cite{acf2019} becomes (\ref{bound:secondchi}). In the next theorem we show that such inequality also holds for a general graph.

\begin{theorem}
Let $G$ be a graph with $n$ vertices, adjacency matrix $A$, and eigenvalues
$\lambda_1\ge\cdots \ge \lambda_n$. Let $p_k\in \Real_k[x]$ such that $p_k(\lambda_1) > p_k(\lambda_i)$ for $i=2,\ldots,n$. Then,
\begin{equation}\label{bound:secondchi}
\chi_k \geq \frac{p_k(\lambda_1)-\lambda(p_k)}{W(p_k)-\lambda(p_k)}.
\end{equation}
\end{theorem}

\begin{proof}
 The proof uses an argument which follows the main lines of reasoning as Haemers does for deriving a lower bound for $\chi$ of any graph in \cite[Th. 4.1 (i)]{H1995}. However, as the last steps are different, we include the complete proof.
Let $\vecnu=(\nu_1,\ldots, \nu_n)$  be the (positive) Perron (column) $\lambda_1$-eigenvector of $A$. Let $V_1,\ldots, V_{\chi_k}$ be the color classes of $G^k$. Let $\tilde{S}$ be the $n\times \chi_k$ matrix with entries
$$
(\tilde{S})_{uj}=\left\{
\begin{array}{ll}
\nu_u, & \mbox{if $u\in V_j$},\\
0, & \mbox{otherwise.}
\end{array}
\right.
$$
Notice that, with the appropriate length of the  vector $\1$, we have
$$
\textstyle
\tilde{S}\1=\vecnu\quad \mbox{and} \quad \tilde{S}^{\top}\vecnu=(\sum_{u\in V_1}\nu_u^2,\ldots,\sum_{u\in V_{\chi_k}}\nu_u^2 )^{\top}.
$$
Let $S$ be the matrix $\tilde{S}$ with all its normalized column vector. That is, $S=\tilde{S}D^{\frac{1}{2}}$ where $D=\tilde{S}^{\top}\tilde{S}=\diag(\sum_{u\in V_1}\nu_u^2,\ldots,\sum_{u\in V_{\chi_k}}\nu_u^2 )$.
Now consider the $\chi_k\times \chi_k$ matrix $B=S^{\top}p_k(A) S$ which, as it is readily checked by using the above, has eigenvalue $p_k(\lambda_1)$ with eigenvector $D^{\frac{1}{2}}\1$. Moreover, since each principal submatrix of $B$ corresponding to a color class has all its off-diagonal entries equal to zero, we have
\begin{align*}
(B)_{ii} & =\sum_{u\in V_i} (S^{\top})_{iu} (p_k(A))_{uu} (S)_{ui}=\sum_{u\in V_i} (p_k(A))_{uu}\frac{\nu_u^2}{\sum_{v\in V_{i}}\nu_v^2}\\
  & \le W(p_k )\frac{1}{\sum_{v\in V_{i}}\nu_v^2} \sum_{u\in V_i}\nu_u^2=W(p_k), \qquad i=1,\ldots, \chi_k.\\
\end{align*}
Besides, by using interlacing, all the eigenvalues of $B$ must be  between $\lambda(p_k)$ and $p_k(\lambda_1)$. Hence,
$$
\chi_k W( p_k)\ge \sum_{i=1}^{\chi_k}(B)_{ii}=\tr(B)\ge p_k(\lambda_1)+(\chi_k-1)\lambda(p_k)
$$
and the result follows.
\end{proof}

\subsection{Second ratio-type bound for $\chi_k$}

In this section we extend the algebraic bound for $\chi$ by Haemers \cite[Th. 4.1$(ii)$]{H1995} to the distance chromatic number.
\begin{theorem}
\label{theo4.1iiHaemers:extended}
Let $G$ be a $k$-partially walk-regular graph with adjacency matrix eigenvalues $\lambda_1\geq \cdots \geq \lambda_n$. Let $p_k\in \mathbb{R}_k[x]$ such that $\sum_{i=1}^{n} p_k(\lambda_i) = 0$, and let $\Phi_1 \geq\Phi_2 \geq\cdots \geq \Phi_n$ be the eigenvalues of $p_k(A)$. If $\Phi_2>0$, then

\begin{equation}\label{extendedHaemersthm4.1ii}
\chi_k\geq 1-\frac{\Phi_{n-\chi_k+1}}{\Phi_2}.
 \end{equation}
  \end{theorem}
 \begin{proof}
An analogous interlacing argument as the one used in as in \cite[Th. 4.1 (ii)]{H1995} applies here, where instead of the adjacency matrix $A$ and the quotient matrix $B$, now we consider linear combinations of both matrices, $p_k(A)$ and $p_k(B)$.
 \end{proof}

\section{Concluding remarks}

We should note that computing our eigenvalue bounds (using the MILPs) is significantly faster than solving the SDP of the Lov\'asz theta bound, and in many cases our bounds perform fairly good, as shown in Table \ref{tableMILP}.

The optimization of the first inertial-type bound (\ref{eq:thm1}) using the MILP \eqref{MILP:inertia}   has special interest since our first inertial-type bound (\ref{eq:thm1}) provide an upper bound for the quantum $k$-independence number [Theorem 7, \cite{wea2020}], which is, in general, not known to be a computable parameter.

While for distance-regular graphs one can use the celebrated linear programming bound by Delsarte \cite{d1973} on $G^k$ in order to bound $\alpha_k$, our inertial-type bound (\ref{eq:thm1}) and its MILP \eqref{MILP:inertia} are more general since they can also be applied to vertex-transitive graphs which are not distance-regular, or in general, to walk-regular graphs which are not distance-regular.

For walk-regular graphs, it is expected that our first inertial bound implementation \eqref{MILP:inertiaWR} does not outperform the ratio-type bound implemented using the so-called minor polynomials \cite{fiol20}. This is due to the fact that our MILP \eqref{MILP:inertiaWR} uses a linear combination of the eigenvalue multiplicities which is more restrictive than the multiplicity linear combination used with the minor polynomials. However, our first inertial-type bound implementation with the MILP \eqref{MILP:inertia} is more general than the ratio-type bound implementation from \cite{fiol20}, since the latter requires walk-regularity while our first inertial-type bound (\ref{eq:thm1}) and its MILP \eqref{MILP:inertia} apply to general graphs.

We end with two open problems that we feel are most natural to try
next. The same MILP method as we use in Sections \ref{subsec:opt-first} and \ref{optsecondinertialbound} could be useful to find the target polynomial in other graphs and/or for other values of $k$. Some graph candidates would be vertex-transitive graphs which are not distance-regular (since otherwise one can just use Delsarte LP bound). Finally, note that our MILP formulations to optimize the spectral bounds for $\alpha_k$ and $\chi_k$ have a polynomial number of input variables, hence it would be interesting to study whether these formulations admit an algorithm in polynomial time \cite{lenstra}.

\section*{Acknowledgments}
The research of A. Abiad is partially supported by the FWO grant 1285921N. A. Abiad and M.A. Fiol gratefully acknowledge the support from DIAMANT. B. Nogueira acknowledges grant PRPG/ADRC from UFMG. The authors would also like to thank Anurag Bishnoi for noticing a tight family for our bound \eqref{inertialikeboundk2distancechromaticnumber}.


\begin{thebibliography}{10}
	

\bibitem{act2016}
A. Abiad, S.M Cioab\u{a} and M. Tait, Spectral bounds for the $k$-independence number of a graph, \emph{Linear Algebra Appl.} {\bf 510} (2016), 160--170.

\bibitem{adf2016}
A. Abiad, E.R. van Dam and M.A. Fiol, Some Spectral and Quasi-Spectral Characterizations of Distance-Regular Graphs, \emph{J. Combin. Theory Ser. A} {\bf 143} (2016), 1--18.
	
\bibitem{acf2019}
A. Abiad, G. Coutinho and M.A. Fiol, On the $k$-independence number of graphs,
\emph{Discrete Math.} {\bf 342} (2019), 2875--2885.

\bibitem{amz2020}
A. Abiad, R. Mulas and D. Zhang, Coloring the normalized Laplacian for oriented hypergraphs.  {\tt arXiv:2008.03269}.

\bibitem{am2002}
N. Alon and B. Mohar, The chromatic number of graph powers, \emph{Combin. Probab. Comput.} {\bf 11}  (2002), 1--10.
	
\bibitem{AF2004}
G. Atkinson and A. Frieze, On the $b$-Independence Number of Sparse Random Graphs, {\em Combin. Probab. Comput.} {\bf 13} (2004), 295--309.

\bibitem{bdz2005}
M. Beis, W. Duckworth and M. Zito,
Large k-independent sets of regular graphs,
\emph{Electron. Notes in Discrete Math.}  {\bf 19} (2005), 321--327.

\bibitem{biggs}
N.L. Biggs,
{\em Algebraic Graph Theory}, 2nd ed., Cambridge University Press,  Cambridge, 1993.



\bibitem{CDS}
F. Chung, C. Delorme and P. Sol\'e, Multidiameters and Multiplicities, \emph{Europ. J. Combinatorics}  {\bf 20} (1999), 629--640.

\bibitem{c71}
D.M. Cvetkovi\'c, Graphs and their spectra, {\em Publ. Elektrotehn. Fak. Ser. Mut. Fiz.} {\bf 354-356} (1971) 1--50.

\bibitem{ddfgg11}
C. Dalf\'o, E.R. van Dam, M.A. Fiol, E. Garriga and B.L. Gorissen,
On almost distance-regular graphs, {\em J. Combin. Theory, Ser. A} {\bf 118} (2011), 1094--1113.

\bibitem{d1973}
P. Delsarte, An algebraic approach to the association schemes of coding theory, Philips Res. Rep., Suppl. 10 (1973).


\bibitem{DZ2003}
W. Duckworth and M. Zito, Large $2$-independent sets of regular graphs, {\em Electron. Notes in Theo. Comp. Sci.} \textbf{78} (2003), 223--235.

\bibitem{elphick}
C. Elphick and P. Wocjan, An inertial lower bound for the chromatic number of  a graph, \emph{Electron. J. Combin.} {\bf 24(1)} (2017), \#P1.58.

\bibitem{en}
K. Enami and S. Negami. Recursive Formulas for Beans Functions of Graphs, \emph{Theory and Applications of Graphs} 7 (2020).


\bibitem{FGY1997}
M.A. Fiol, E. Garriga and J.L.A. Yebra, The alternating polynomials and their relation with the spectra and conditional diameters of graphs, \emph{Discrete Math.} \textbf{167-168} (1997), 297--307.

\bibitem{f99}
M.A. Fiol,
Eigenvalue interlacing and weight parameters of graphs,
{\it Linear Algebra Appl.} {\bf 290} (1999), 275--301.

\bibitem{f02}
M.A. Fiol,
Algebraic characterizations of distance-regular graphs,
{\it Discrete Math.} {\bf 246(1-3)} (2002), 111--129.

\bibitem{f16}
M.A. Fiol,
Quotient polynomial graphs, {\em Linear Algebra Appl.} {\bf 488} (2016), 363--376.
	
\bibitem{fiol20}
M.A. Fiol,
A new class of polynomials from the spectrum of a graph, and its application to bound the $k$-independence number, {\em Linear Algebra Appl.} \textbf{605} (2020), 1--20.

\bibitem{fg97}
M.A. Fiol and E. Garriga, From local adjacency
polynomials to locally pseudo-distance-regular graphs, {\it
J.  Combin. Theory Ser. B} {\bf 71} (1997), 162--183.

\bibitem{FG1998}
M.A. Fiol and E. Garriga, The alternating and adjacency polynomials, and their
relation with the spectra and diameters of graphs, \emph{Discrete Appl. Math.} \textbf{87} (1998) 77--97.

\bibitem{FGY1996}
M.A. Fiol, E. Garriga and J.L.A. Yebra, Locally pseudo-distance-regular
graphs, \emph{J. Combin. Theory Ser. B} \textbf{68} (1996), 179--205.

\bibitem{fh97}
P. Firby and J. Haviland, Independence and average distance in graphs,
{\it Discrete Appl. Math.} {\bf 75} (1997), 27--37.

\bibitem{g87}
J. Gallian,
Labeling prisms and prism related graphs {\em Congr. Numer.} {\bf 59} (1987), 89--100.

\bibitem{GHHHR2008}
W. Goddard, S.M. Hedetniemi, S.T. Hedetniemi, J.M. Harris and D.F. Rall, Broadcast chromatic numbers of graphs,
{\em Ars Combin.} {\bf 86} (2008), 33--49.

\bibitem{gm80}
C.D. Godsil and B.D. Mckay, Feasibility conditions for the existence of walk-regular graphs, {\em Linear Algebra Appl.} {\bf 30} (1980), 51--61.

\bibitem{H1995}
W.H. Haemers, Interlacing eigenvalues and graphs,
\emph{Linear Algebra Appl.} {\bf 226-228} (1995), 593--616.
	
\bibitem{HKSS2002}
G. Hahn, J. Kratochv\'il, J. \v{S}ir\'a\v{n} and D. Sotteau,
On the injective chromatic number of graphs, {\em Discrete Math.} \textbf{256} (2002), 179-192.

\bibitem{h86}
M. Hall, {\em Combinatorial Theory}, 2nd edition,  John Wiley \& Sons, Inc., New York, 1986.

\bibitem{H2003}
J. van den Heuvel and S. McGuinnes, Colouring the square of a planar graph, \emph{J. Graph Theory} {\bf 42} (2003), 110--124.

\bibitem{HPP2001}
M. Hota, M. Pal and T.K. Pal, An efficient algorithm for finding a maximum weight $k$-independent set on trapezoid graphs, \emph{Comput. Optim. Appl.} \textbf{18} (2001), 49--62.

\bibitem{jt1995}
T.R. Jensen and B. Toft, {\em Graph Coloring Problems}, Wiley, New York, 1995.

\bibitem{jo1996}
A. Johansson,
Asymptotic choice number for triangle-free graphs. Technical Report {\bf 91-5}, DIMACS, 1996.	

\bibitem{JLL2020}
M.-J. Jou, J.-J. Lin and Q.-Y. Lin, On the 2-Independence Number of Trees, \emph{International Journal of Contemporary Mathematical Sciences}  {\bf 15} (2020), 107--112.

\bibitem{kp2016}
R.J. Kang and F. Pirot, Coloring powers and girth, \emph{SIAM Journal on Discrete Mathematics}  {\bf 30 (4)} (2016), 1938--1949.

\bibitem{kp2018}
R.J. Kang and F. Pirot,  Distance colouring without one cycle length,  \emph{Combin. Probab. Comput.} {\bf 27(5)} (2018), 794--807.

\bibitem{kz1993}
M.C. Kong and Y. Zhao, On computing maximum $k$-independent sets, \emph{Congr. Numer.} {\bf 95} (1993), 47–60.

\bibitem{kz2000}
M.C. Kong and Y. Zhao, Computing k-Independent sets for regular bipartite graphs, \emph{Congr. Numer.}  {\bf 143} (2000), 65--80.
	
\bibitem{21}
F. Kramer, Sur le nombre chromatique $K(p,G)$ des graphes, RAIRO R-1 (1972), 67--70.
	
\bibitem{survey}
F. Kramer and H. Kramer, A survey on the distance-colouring of graphs, \emph{Discrete Math.} {\bf 308} (2008), 422--426.

\bibitem{LUW1995}
F. Lazebnik and V. Ustimenko and A. Woldar, A New Series of Dense Graphs of High Girth, \emph{Bull. AMS} {\bf 32} (1995), 73--79.

\bibitem{lenstra}
H.W. Lenstra, Integer programming with a fixed number of variables, \emph{Mathematics of Operations Research} \textbf{8(4)} (1983), 538--548.
	
	

\bibitem{l79}
L. Lov\'asz, On the Shannon capacity of a graph, \emph{IEEE Trans. on Inform. Theory} \textbf{25} (1979), 1--7.

\bibitem {MWS}
F.J. MacWilliams and N.J. Sloane, {\em The Theory of Error-correcting Codes}, North Holland, New York, 1981.
	
\bibitem {MancinskaRoberson}
L. Mančinska and D.E. Roberson, Quantum homomorphisms, {\em Journal of Combinatorial Theory B} {\bf 118} (2016), 228--267.		
	
\bibitem{M2000}
M. Mahdian, {\em The Strong Chromatic Index of Graphs}, M.Sc. Thesis, University of Toronto (2000).	

\bibitem{N}
T. Nierhoff, {\em The $k$-Center Problem and $r$-Independent Sets}, PhD. Thesis, Humboldt University, Berlin.
	
\bibitem{OShiTaoqiu2019}
S. O, Y. Shi and Z. Taoqiu, Sharp upper bounds on the $k$-independence number in regular graphs, {\tt arXiv:1901.06607}.

\bibitem{s56}
C.E. Shannon, The zero-error capacity of a noisy channel, {\em IRE Trans. Inform. Theory} {\bf 2} (1956), no. 3, 8--19.

\bibitem{glasgow}
Symmetric 2-designs, Web page Glasgow University, at\newline
{\tt http://www.maths.gla.ac.uk/$\thicksim$es/symmdes/2designs.php}
	
	
\bibitem{wea2020}
P. Wocjan, C. Elphick and A. Abiad, Spectral upper bound on the quantum $k$-independence number of a graph, {\tt arXiv:1910.07339}.	
	
	
\bibitem{ew2013}
P. Wocjan and C. Elphick, New spectral bounds on the chromatic number encompassing all eigenvalues of the adjacency matrix, {\em Electron. J. Combin.} {\bf  20(3)} (2013), \#P39.

\bibitem{yfma85}
J.L.A. Yebra, M.A. Fiol, P. Morillo and I. Alegre,
The diameter of undirected graphs associated to plane tessellations,
{\it Ars Combin.} {\bf 20B} (1985), 159--171.
\end{thebibliography}
\end{document}